\documentclass[a4paper]{article}

\usepackage[a4paper,margin=1in]{geometry}
\usepackage{amsmath}
\usepackage{amsfonts}
\usepackage{hyperref}
\usepackage[all]{xy}
\usepackage{amsthm}

\newtheorem{theorem}{Theorem}

\newtheorem{definition}[theorem]{Definition}
\newtheorem{remark}[theorem]{Remark}
\newtheorem{example}[theorem]{Example}

\title{\textbf{Connection Cochain of Abelian Extensions and Connection $1$-Forms}}
\author{DING BOSHU}

\begin{document}

\maketitle

\begin{abstract}
   \noindent  In this paper, we consider the concept of connection cochain of central extensions introduced by Moriyoshi and apply it to the abelian case. We will show the relationship between connection cochain and connection $1$-form of a principal bundle whose structure group is abelian.
\end{abstract}

\section{Introduction}

In [1], Moriyoshi has introduced the concept of the \textbf{connection cochain} with respect to central extensions. 
\begin{definition} \normalfont \cite{moriyoshi}
   Suppose that there is a central extension $\xymatrix{1\ar[r] &  A \ar[r] &  \tilde G \ar[r] &  G \ar[r] &  1}$. A cochain $\tau:G\to A$ that satisfies the condition
   \begin{align*}
     \tau(ga)=\tau(g)+a
   \end{align*}
   for $g\in G$ and $a\in A$ is called a \textbf{connection cochain}.
\end{definition}
\noindent In Definition 3, we apply this idea to the abelian extensions. From Remark 4 to Remark $7$, we review some basic concepts of abelian extensions. In Example 8, we study the abelian extension
\begin{align*}
         \xymatrix{
            1 \ar[r]  &   C^\infty(P,G)^G \ar[r]^-\iota  &   \text{Aut}(P)  \ar[r]^-\rho  &  \text{Diff}(M)_P \ar[r]  &  1 
         }
     \end{align*}
     corresponding to a principal $G$-bundle, where $G$ is an abelian group. We give an explicit expression of the action of $\text{Diff}(M)_P$ on $C^\infty(P,G)^G$ in Equation (4), and the corresponding infinitesimal action in Equation (5). With the constructions above, we show
     \begin{theorem}[Theorem 9]
\normalfont
  Let $\tau:\text{Aut}(P)\to C^\infty(P,G)^G$ be a conenction cochain of abelian extension with respect to (3) and $\tau_{*,1_{\text{Aut}(P)}}:\mathfrak X(P)^r\to C^\infty(P,\mathfrak g)^G$ its differential map over $1_{\text{Aut}(P)}$. If for all $X\in\mathfrak X(P)^r$ with $X_p=0$, 
  \begin{align}
      \tau_{*,1_{\text{Aut}(P)}}(X)(p)=0,
  \end{align}
  then ${l_{\tau(1_{\text{Aut}(P)})^{-1}}}_* \tau_{*,1_{\text{Aut}(P)}}$ is a connection $1$-form over $P$.
\end{theorem}

\section{Preliminary}

\begin{definition}
\normalfont For an abelian group extension
  \begin{align}
     \xymatrix{
         1 \ar[r] &   A \ar[r] &   \tilde G \ar[r]  &   G \ar[r]  &    1
       }
  \end{align}
  a \textbf{connection cochain of abelian extensions} is a map $\tau:\tilde G\to A$ such that
  \begin{align*}
         \tau(\tilde ga)=\tau(\tilde g)+a,
  \end{align*}
  for all $\tilde g\in \tilde G$ and $a\in A$.
\end{definition}

\begin{remark}
\normalfont
 Define $\mu:G\to\text{Diff}(A)$ by
\begin{align}
       \mu(g)a=\tilde g^{-1}a\tilde g,
\end{align}
where $a\in A$, $g\in G$ and $\tilde g$ is a lift of $g$ in $\tilde G$. Note that (3) is independent of the choice of the lifts. Indeed, let $\tilde g'$ be another lift of $g$. Then, since
\begin{align*}
   \pi(\tilde g'\tilde g^{-1})=gg^{-1}=1_G.
\end{align*}
we see that there exists $a'\in A$ such that 
\begin{align*}
   \tilde g'\tilde g^{-1}=a'.
\end{align*}
It follows that 
\begin{align*}
   \tilde g'^{-1}a\tilde g'=\tilde g^{-1}a'^{-1}aa'\tilde g=\tilde g^{-1}a\tilde g.
\end{align*}
It's easy to see that $\mu$ induces an action of $G$ on $A$. When (2) is a Lie group extension, then $\mu$ is a smooth action. Similarly, we can define an action of $\tilde G$ on A given by
\begin{align}
     \nu(\tilde g)a=\tilde g^{-1}a\tilde g.
\end{align}
These action (3) and (4) induces the coboundary map of the Lie group cochains $C^*(G,A)$ and $C^*(\tilde G,A)$, respectively.
\end{remark}

\begin{remark}
  \normalfont
   If (2) is a Lie group extension, then there exists the following corresponding Lie algebra extension
   \begin{align*}
     \xymatrix{
     0 \ar[r]  &     \mathfrak a  \ar[r]  &  \tilde{\mathfrak g} \ar[r]  &  \mathfrak g  \ar[r]  & 0
     }
   \end{align*}
   The infinitesimal action of (3) is the differential $\mu_{*,1_G}:\mathfrak g\to\mathfrak X(A)$ given by
   \begin{align*}
        \mu_{*,1_G}(V)a=\frac d{dt}\bigg|_{t=0}\mu(c_V(t))a
   \end{align*}
   where $V\in \mathfrak g$, $a\in A$ and $c_V(t)$ is a curve in $G$ with initial vector $V$. Note that the curve induced by the exponential map $\exp(tV)$ may be a choice of $c_V(t)$. Similarly, we have the action $\nu_{*,1_{\tilde G}}$ corresponding to $\nu$. These actions may be used to give the boundary map of the associated Lie algebra cohomology.
\end{remark}

\begin{remark}
   \normalfont
   The abelian extension (2) induces a $2$-cocycle $\chi:G\times G\to A$ associated with a section $s:G\to \tilde G$, given by
   \begin{align*}
        (g_1,g_2)\mapsto s(g_1g_2)^{-1}s(g_1)s(g_2),
   \end{align*}
   called the \textbf{Euler cocycle}. The terminology $2$-cocycle means that
   \begin{align*}
          \chi(g_2,g_3)-\chi(g_1g_2,g_3)+\chi(g_1,g_2g_3)-\mu(g_3)\chi(g_1,g_2)=0.
   \end{align*}
\end{remark}

\begin{remark}
   \normalfont
    Let $G\times_\chi A$ be the group with multiplication
    \begin{align*}
         (g_1,a_1)(g_2,a_2)=(g_1g_2,\chi(g_1,g_2)+\mu(g_2)a_1+a_2),
    \end{align*}
    the group structure of which is ensured by the fact that $\chi$ is a $2$-cocycle. It's easy to see that $G\times_\chi A$ is isomorphic to $\tilde G$. Indeed, define $\Phi:G\times_\chi A\to\tilde G$ by
    \begin{align*}
        (g,a)\mapsto s(g)a
    \end{align*}
    where $s:G\to \tilde G$ is the section inducing the Euler cocycle. Since
    \begin{align*}
       \Phi((g_1,a_1)(g_2,a_2))&=\Phi(g_1g_2,\chi(g_1,g_2)+\mu(g_2)a_1+a_2)\\
       &=s(g_1g_2)s(g_1g_2)^{-1}s(g_1)s(g_2)s(g_2)^{-1}a_1s(g_2)a_2  \\
       &=s(g_1)a_1s(g_2)a_2 \\
       &=\Phi(g_1,a_1)\Phi(g_2,a_2),
    \end{align*}
    we see that $\Phi$ is a homomorphism. On the other hand, define $\Psi:\tilde G\to G\times_\chi A$ by
    \begin{align*}
       \tilde g\mapsto (g,s(g)^{-1}\tilde g).
    \end{align*}
    Since
    \begin{align*}
    \Phi\circ\Psi(\tilde g)=\Phi(g,s(g)^{-1}\tilde g)=s(g)s(g)^{-1}\tilde g=\tilde g,
    \end{align*}
    and
    \begin{align*}
       \Psi\circ\Phi(g,a)=\Psi(s(g)a)=(g,a),
    \end{align*}
    we deduce that $\Phi$ is an isomorphism.
\end{remark}

\begin{example}
\normalfont
     Let $G$ be an abelian Lie group and $\pi:P\to M$ a principal $G$-bundle. Then, we have the following abelian Lie group extension
     \begin{align}
         \xymatrix{
            1 \ar[r]  &   C^\infty(P,G)^G \ar[r]^-\iota  &   \text{Aut}(P)  \ar[r]^-\rho  &  \text{Diff}(M)_P \ar[r]  &  1 
         }
     \end{align}
     where 
     \begin{align*}
     C^\infty(P,G)^G=\{\gamma\in C^\infty(P,G)\mid \gamma(pg)=\gamma(p)\text{ for all $p\in P$ and $g\in G$}\},
     \end{align*}
     and 
     \begin{align*}
     \text{Aut}(P)=\{\varphi\in \text{Diff}(P)\mid \varphi(pg)=\varphi(p)g\text{ for all $p\in P$ and $g\in G$}\}.
     \end{align*}
     Moreover, $\iota:C^\infty(P,G)^G\to\text{Aut}(P)$ is given by
     \begin{align*}
          \iota(\gamma)(p)=p\gamma(p)
     \end{align*}
    for $p\in P$, and $\rho:\text{Aut}(P)\to\text{Diff}(M)_P$ is given by
     \begin{align*}
         \rho(\varphi)(x)=\pi(\varphi(p)),
     \end{align*}
     where $x\in M$ and $p$ is a lift of $x$ in $P$. Note that $\text{Diff}(M)_P=\text{Im }\rho$. The corresponding Lie algebra extension is
     \begin{align*}
     \xymatrix{
          0  \ar[r]  &    C^\infty(P,\mathfrak g)^G  \ar[r]^-{\iota_*} & \mathfrak X(P)^r  \ar[r]^-{\rho_*}  &   \mathfrak X(M) \ar[r] &  0
          }
     \end{align*}
     where 
     \begin{align*}
         \mathfrak X(P)^r=\{X\in\mathfrak X(P)\mid {r_g}_*X_p=X_{pg}\text{ for all $p\in P$ and $g\in G$}\},
     \end{align*}
     called the \textbf{right-invariant vector field} over $P$. Indeed, for any curve $\varphi(t)$ in $\text{Aut}(P)$ through $1_{\text{Aut}(P)}$, we have
     \begin{align*}
               {r_g}_*\frac d{dt}\bigg|_{t=0}\varphi(p,t)=\frac d{dt}\bigg|_{t=0}\varphi(p,t)g=\frac d{dt}\bigg|_{t=0}\varphi(pg,t),
     \end{align*}
   which implies that $\text{Lie}(\text{Aut}(P))=\mathfrak X(P)^r$. Moreover, 
   \begin{align*}
        C^\infty(P,\mathfrak g)^G=\{\gamma\in C^\infty(P,\mathfrak g)\mid \gamma(pg)=\gamma(p)\text{ for all $p\in P$ and $g\in G$}\},
   \end{align*}
   which is the Lie algebra of $C^\infty(P,G)^G$. The maps $\iota_*$ and $\rho_*$ are the differentials of $\iota$ and $\rho$ over the identities of the corresponding groups, respectively. The action of $\text{Diff}(M)_P$ on $C^\infty(P,G)^G$ is given by
   \begin{align}
   \begin{split}
          (\mu(f)\gamma)(p)&=s(f)^{-1}\circ\iota(\gamma)\circ s(f)(p)=s(f)^{-1}(s(f)(p)\gamma(s(f)(p)))=p\gamma(s(f)(p))\\
          &=\gamma\circ\sigma\circ f\circ\pi(p),
          \end{split}
   \end{align}
   for $f\in\text{Diff}(M)_P$, $\gamma\in C^\infty(P,G)^G$ and $p\in P$, where $s:\text{Diff}(M)_P\to\text{Aut}(P)$ is a section of $\rho:\text{Aut}(P)\to\text{Diff}(M)_P$ and $\sigma:M\to P$ a section of $P$ which is smooth around $f(\pi(p))$. Note that (6) is independent of the choice of $\sigma$ and we have identified $C^\infty(P,G)^G$ with its image in $\text{Aut}(P)$. The infinitesimal action of $\mu$ for $X\in\mathfrak X(M)$, $\gamma\in C^\infty(P,G)^G$, and $p\in P$ is given by
   \begin{align}
   \begin{split}
      (\mu_{*,1_{\text{Aut}(P)}}(X)\gamma)(p)&=\frac d{dt}\bigg|_{t=0}(\mu(\varphi_X(-,t))\gamma)(p)=\frac d{dt}\bigg|_{t=0}\gamma\circ\sigma\circ \varphi_X(-,t)\circ \pi(p)\\
      &=\gamma_{*,\sigma(\pi(p))}\sigma_{*,\pi(p)}(X_{\pi(p)}) \\
      &=\gamma_{*,p}(\tilde X_p),
      \end{split}
   \end{align}
   where $\varphi_X(-,t)$ is the flow of $X$ and $\tilde X_p$ is a lift of $X_{\pi(p)}$ in $T_pP$. We need to show that (7) is independent of the choice of the lift.  Indeed, let $\tilde X'_p$ be another lift in $T_pP$. Then, since
   \begin{align*}
          \gamma_{*,p}(\tilde X_p-\tilde X_p')=\gamma_{*,p}{j_p}_{*,1_G}V=\frac d{dt}\bigg|_{t=0} \gamma(pc_V(t))=\frac d{dt}\bigg|_{t=0}\gamma(p)=0,
   \end{align*}
   where $c_V(t)$ is a curve in $G$ with initial vector $V$, we see that (7) is independent of the choice of the lift with respect to a fixed $p$. Note that since $\tilde X_p-\tilde X'_p$ is vertical, there exists $V\in\mathfrak g$ such that ${j_p}_{*,1_G}V=\tilde X_p-\tilde X_p'$, where $j_p:G\to P$ is given by
   \begin{align*}
      g\mapsto pg.
   \end{align*}
   If $\tilde X''_{pg}$ is another lift in $T_{pg}P$, then since ${r_g}_*\tilde X_p$ is also a lift of $X$ in $T_{pg}P$, we see that
   \begin{align*}
     \gamma_{*,pg}(\tilde X''_{pg})=\gamma_{*,pg}({r_g}_*\tilde X_p)=\frac d{dt}\bigg|_{t=0} \gamma(c_{\tilde X_p}(t)g)=\frac d{dt}\bigg|_{t=0}\gamma(c_{\tilde X_p}(t))=\gamma_{*,p}(\tilde X_p),
   \end{align*}
   where $c_{\tilde X_p}(t)$ is a curve in $P$ with initial vector $\tilde X_p$. Therefore, (5) is independent of the choice of the lift.
\end{example}

\section{Theory}

\begin{theorem}
\normalfont
  Let $\tau:\text{Aut}(P)\to C^\infty(P,G)^G$ be a conenction cochain of abelian extension with respect to (3) and $\tau_{*,1_{\text{Aut}(P)}}$ its differential map over $1_{\text{Aut}(P)}$. If for all $X\in\mathfrak X(P)^r$ with $X_p=0$, 
  \begin{align}
      \tau_{*,1_{\text{Aut}(P)}}(X)(p)=0,
  \end{align}
  then ${l_{\tau(1_{\text{Aut}(P)})^{-1}}}_* \tau_{*,1_{\text{Aut}(P)}}$ is a connection $1$-form over $P$.
\end{theorem}

\begin{proof}
  Recall that any tangent vector $X_p\in T_pP$ could be extended to a right-invariant vector field $X$ over $P$. Then, ${l_{\tau(1_{\text{Aut}(P)})^{-1}}}_* \tau_{*,1_{\text{Aut}(P)}}$ is regarded as a $\mathfrak g$-valued in the sense that
  \begin{align*}
        {l_{\tau(1_{\text{Aut}(P)})^{-1}}}_* \tau_{*,1_{\text{Aut}(P)}}(X_p)={l_{\tau(1_{\text{Aut}(P)})^{-1}}}_* \tau_{*,1_{\text{Aut}(P)}}(X)(p).
  \end{align*}
  The condition (8) ensures that it is well-defined. Now, we check that ${l_{\tau(1_{\text{Aut}(P)})^{-1}}}_* \tau_{*,1_{\text{Aut}(P)}}$ is a connection $1$-form. Let $V\in\mathfrak g$ and $\underline V$ its fundamental vector field over $P$. Since $G$ is abelian, we see that $\underline V$ is rignt-invariant. It follows that
  \begin{align*}
    {l_{\tau(1_{\text{Aut}(P)})^{-1}}}_* \tau_{*,1_{\text{Aut}(P)}}(\underline V_p)={l_{\tau(1_{\text{Aut}(P)})^{-1}}}_* \tau_{*,1_{\text{Aut}(P)}}(\underline V)(p)=\frac d{dt}\bigg|_{t=0}\tau(1_{\text{Aut}(P)})^{-1} \tau(\varphi_{\underline V}(-,t))(p),
  \end{align*}
  where $\varphi_{\underline V}(-,t)$ is the flow of $\underline V$. It's easy to see that
  \begin{align*}
    \varphi_{\underline V}(-,t)=(-)\exp(tV),
  \end{align*}
  since
  \begin{align*}
        \frac d{dt}\bigg|_{t=s} (-)\exp(tV)=\frac d{dt}\bigg|_{t=0} (-)\exp(sV)\exp(tV)=\underline V_{(-)\exp(sV)}.
  \end{align*}
  It follows that
  \begin{align*}
      {l_{\tau(1_{\text{Aut}(P)})^{-1}}}_* \tau_{*,1_{\text{Aut}(P)}}(\underline V_p)&=\frac d{dt}\bigg|_{t=0}\tau(1_{\text{Aut}(P)})^{-1} \tau((-)\exp tV)(p) \\
      &=\frac d{dt}\bigg|_{t=0} \tau(1_{\text{Aut}(P)})^{-1}\tau(1_{\text{Aut}(P)}\circ\iota(\exp(tV)))(p),
  \end{align*}
  where we have regarded $\exp(tV)$ as a constant function in $C^\infty(P,G)^G$ and $\iota(\exp(tV))(-)=(-)\exp(tV)$. Thus, by the definition of connection cochain, we have
  \begin{align*}
       {l_{\tau(1_{\text{Aut}(P)})^{-1}}}_* \tau_{*,1_{\text{Aut}(P)}}(\underline V_p)&=\frac d{dt}\bigg|_{t=0} \tau(1_{\text{Aut}(P)})^{-1}\tau(1_{\text{Aut}(P)})\exp(tV)(p) \\
       &=\frac d{dt}\bigg|_{t=0}\exp(tV)(p) \\
       &=\frac d{dt}\bigg|_{t=0}\exp(tV) \\
       &=V.
  \end{align*}
  Moreover, we have
  \begin{align*}
      {l_{\tau(1_{\text{Aut}(P)})^{-1}}}_* \tau_{*,1_{\text{Aut}(P)}}({r_g}_*X_p)&= {l_{\tau(1_{\text{Aut}(P)})^{-1}}}_* \tau_{*,1_{\text{Aut}(P)}}(X)(pg)\\
      &= {l_{\tau(1_{\text{Aut}(P)})^{-1}}}_* \tau_{*,1_{\text{Aut}(P)}}(X)(p)\\
      &= {l_{\tau(1_{\text{Aut}(P)})^{-1}}}_* \tau_{*,1_{\text{Aut}(P)}}(X_p)
  \end{align*}
  where $X$ is the right-invariant vector field extending $X_p$, which implies that $X_{pg}={r_g}_*X_p$. Therefore,  ${l_{\tau(1_{\text{Aut}(P)})^{-1}}}_* \tau_{*,1_{\text{Aut}(P)}}$ is a connection $1$-form.
\end{proof}

\end{document}